\documentclass{amsart}

\usepackage{amsfonts,amsmath,amssymb,amsthm,amscd,amsxtra}
\usepackage{enumerate,verbatim}
\usepackage[usenames,dvipsnames]{pstricks}
\usepackage[mathscr]{eucal}

\newcommand{\numberseries}{\bfseries}   

\newlength{\thmtopspace}                
\newlength{\thmbotspace}                
\newlength{\thmheadspace}               
\newlength{\thmindent}                  

\newtheoremstyle{fixed bf head,slanted body}
                {\thmtopspace}{\thmbotspace}{\slshape}
                {\thmindent}{\bfseries}{.}{\thmheadspace}
                {{\numberseries \thmnumber{#2\;}}\thmname{#1}\thmnote{ (#3)}}

\newtheoremstyle{variable bf head,slanted body}
                {\thmtopspace}{\thmbotspace}{\slshape}
                {\thmindent}{\bfseries}{.}{\thmheadspace}
                {{\numberseries \thmnumber{#2\;}}\thmname{#1}\thmnote{ #3}}

\newtheoremstyle{fixed bf head,upright body}
                {\thmtopspace}{\thmbotspace}{\upshape}
                {\thmindent}{\bfseries}{.}{\thmheadspace}
                {{\numberseries \thmnumber{#2\;}}\thmname{#1}\thmnote{ (#3)}}

\newtheoremstyle{numbered paragraph}
                {\thmtopspace}{\thmbotspace}{\upshape}
                {\thmindent}{\upshape}{}{\thmheadspace}
                {{\numberseries \thmnumber{#2.}}}

\usepackage[notcite, notref]{}
\usepackage[usenames,dvipsnames]{pstricks}
\usepackage[mathscr]{eucal}
\usepackage{amsfonts,amsmath,amssymb,amsthm,amscd,amsxtra}
\usepackage{enumerate,verbatim}
\usepackage[all,2cell,ps]{xy}
\usepackage{mathptmx}
\usepackage[notcite, notref]{}
\usepackage[pagebackref]{hyperref}
\usepackage{todonotes}

\usepackage{nicefrac}
\usepackage{array}
\usepackage{xcolor}

\usepackage[leqno]{amsmath}
\usepackage[pagebackref]{hyperref}

\usepackage[normalem]{ulem}

\makeatletter
\def\UL@putbox{\ifx\UL@start\@empty \else 
  \vrule\@width\z@ \LA@penalty\@M
  {\UL@skip\wd\UL@box \UL@leaders \kern-\UL@skip}%
    \phantom{\box\UL@box}%
  \fi}
\makeatother

\DeclareMathOperator{\md}{\operatorname{\mathsf{mod}}}

\theoremstyle{plain} 

\newtheorem{thm}{Theorem}[section]

\newtheorem{prop}[thm]{Proposition}
\newtheorem{cor}[thm]{Corollary}

\newtheorem{conjecture}[thm]{Conjecture}

\newtheorem{lem}[thm]{Lemma}

\theoremstyle{definition}

\newtheorem{eg}[thm]{Example}
\newtheorem{ques}[thm]{Question}

\newtheorem{rmk}[thm]{Remark}

\numberwithin{equation}{section}
 
\newcommand{\lra}{\longrightarrow}

\def\PF{\mathsf{PF}}

\newcommand{\fm}{\mathfrak{m}}
\newcommand{\fn}{\mathfrak{n}}
\newcommand{\fp}{\mathfrak{p}}
\newcommand{\fq}{\mathfrak{q}}

\newcommand{\CC}{\mathbb{C}}

\newcommand{\ZZ}{\mathbb{Z}}

\newcommand{\QQ}{\mathbb{Q}}

\newcommand{\up}[1]{{{}^{#1}\!}}
\newtheorem{chunk}[thm]{\hspace*{-1.065ex}\bf}

\DeclareMathOperator{\Ass}{Ass}

\def\CI{\operatorname{\mathsf{CI-dim}}}
\def\Gdim{\operatorname{\mathsf{G-dim}}}
\def\AB-dim{\operatorname{\mathsf{H-dim}}}
\def\cx{\operatorname{\mathsf{cx}}}
\def\pd{\operatorname{\mathsf{pd}}}
\def\depth{\operatorname{\mathsf{depth}}}
\def\edim{\operatorname{\mathsf{embdim}}}

\def\Hom{\operatorname{\mathsf{Hom}}}

\def\Tr{\mathsf{Tr}\hspace{0.01in}}

\def\Ext{\operatorname{\mathsf{Ext}}}
\def\Spec{\operatorname{\mathsf{Spec}}}
\def\Tor{\operatorname{\mathsf{Tor}}}

\def\G{\operatorname{\mathsf{G}}}
\def\dim{\operatorname{\mathsf{dim}}}

\def\p{\mathfrak{p}}

\DeclareMathOperator{\coker}{coker}

\DeclareMathOperator{\Min}{Min}
\def\PF{\mathsf{PF}}

\def\len{\operatorname{\mathsf{length}}}

\DeclareMathOperator{\Supp}{Supp}

\newcommand{\tf}[2]{{\boldsymbol\bot}_{#1}{#2}}
\newcommand{\tp}[2]{{\boldsymbol\top}_{\hskip-2pt #1}{#2}}

\def\urltilda{\kern -.15em\lower .7ex\hbox{\~{}}\kern .04em}
\def\urldot{\kern -.10em.\kern -.10em}\def\urlhttp{http\kern -.10em\lower -.1ex
\hbox{:}\kern -.12em\lower 0ex\hbox{/}\kern -.18em\lower 0ex\hbox{/}}

\begin{document}
\vspace*{-0.1ex}

\title[On the vanishing of theta invariant ]{On the vanishing of theta invariant \\ and a conjecture of Huneke and Wiegand}

\author[O. Celikbas]{Olgur Celikbas}
\address{Olgur Celikbas\\
Department of Mathematics \\
West Virginia University\\
Morgantown, WV 26506-6310, U.S.A}
\email{olgur.celikbas@math.wvu.edu}



\maketitle

\setcounter{tocdepth}{1}
\begin{abstract} Huneke and Wiegand conjectured that, if $M$ is a finitely generated, non-free, torsion-free module with rank over a one-dimensional Cohen-Macaulay local ring $R$, then the tensor product of $M$ with its algebraic dual has torsion. This conjecture, if $R$ is Gorenstein, is a special case of a celebrated conjecture of Auslander and Reiten on the vanishing of self extensions that stems from the representation theory of finite-dimensional algebras.

If $R$ is a one-dimensional Cohen-Macaulay ring such that $R=S/(f)$ for some local ring $(S, \mathfrak{n})$, and a non zero-divisor $f \in \mathfrak{n}^2$ on $S$, we make use of Hochster's theta invariant and prove that such $R$-modules $M$ which have finite projective dimension over $S$ satisfy the proposed torsion condition of the conjecture. Along the way we give several applications of our argument pertaining to torsion properties of tensor products of modules.
\end{abstract}


\section{Introduction}

Throughout $R$ denotes a commutative Noetherian local ring with unique maximal ideal $\fm$ and residue field $k$, and $\md R$ denotes the category of all finitely generated $R$-modules. 

The aim of this paper is to study the torsion-freeness property of tensor products of  modules, a subtle topic which mainly stems from the beautiful work of Auslander \cite{Au}. Our focus is on the torsion submodule of tensor products of the form $M\otimes_{R}M^{\ast}$ over one-dimensional Cohen-Macaulay local rings $R$, where $M^{\ast}$ denotes $\Hom_R(M,R)$. In particular, we are concerned with the following long-standing conjecture of Huneke and Wiegand.
\begin{conjecture} \label{HWC} (Huneke and Wiegand \cite[4.6]{HW1}) Let $R$ be a one-dimensional local ring and let $M \in \md R$ be a torsion-free module. Assume $M$ has rank (e.g., $R$ is a domain). If $M\otimes_{R}M^{\ast}$ is torsion-free, then $M$ is free. In other words, if $M$ is not free, then $M\otimes_{R}M^{\ast}$ has torsion, i.e., the torsion submodule of $M\otimes_{R}M^{\ast}$ is not zero.
\end{conjecture}

Recall that a module $M\in \md R$ is said to have \emph{rank} if there is an integer $r$ such that $M_{\fp}\cong R_{\fp}^{\oplus r}$ for all $\p \in \Ass(R)$, where $\Ass(R)$ is the set of all associated primes of $R$. 

Conjecture \ref{HWC} stems from the seminal works of Auslander \cite{Au}, and Huneke and Wiegand \cite{HW1}. The conjecture is true over hypersurface rings \cite[3.7]{HW1}, but it is very much open in general, even for ideals over complete intersection domains of codimension two. It is worth noting that Conjecture \ref{HWC} is a special, and restrictive, version of the celebrated conjecture of Auslander and Reiten \cite{AuRe} on the vanishing of $\Ext$ when the ring in question is a one-dimensional Gorenstein domain; see \cite{CeRo} for details.

There is strong evidence that Conjecture \ref{HWC} should be true over complete intersections; see \cite{CeRo, HSW}. Moreover, there are various examples supporting the conjecture over rings that are not necessarily complete intersections. For example, it  is proved in \cite[3.6]{HSW} that Conjecture \ref{HWC} is true over Cohen-Macaulay rings with minimal multiplicity, e.g., over local \emph{Arf} rings \cite{Lipman}. For some further examples, we refer to \cite{CGTT} and point out the following:

\begin{eg} \label{propA} Let $R$ be a one-dimensional, reduced, non-regular, local ring.
\begin{enumerate} [\rm(i)]
\item If $R$ is complete, and has prime characteristic $p$ and perfect residue field $k$, then $\up{\varphi^n}R\otimes_R (\up{\varphi^n}R)^{\ast}$ has torsion for all $n\gg 0$. Here $\varphi^n: R \to R$ is the $n$th iterate of the \emph{Frobenius} endomorphism given by $r\mapsto r^{p^n}$, and 
$\up{\varphi^n}R$ denotes $R$ with the $R$-action given by $r\cdot s=r^{p^n}s$ for all $r,s \in R$; see \cite[2.15]{CGTT} and \cite[2.1.3 and 2.2.12]{Millercont}.

\item If $R$ is a Gorenstein domain and $I$ is an \emph{Ulrich} ideal of $R$ that contains a parameter ideal as a reduction (e.g., $R = \CC[\![t^4, t^5, t^6]\!]$ and $I=(t^4, t^6)$), then $I$ is a self-dual $R$-module, i.e., $I \cong I^{\ast}$, and so $I\otimes_RI^{\ast}$ has torsion; see Example \ref{ExUlrich} and Remark  \ref{kanitla}. 
\end{enumerate}
\end{eg}

The purpose of this paper is to prove Theorem \ref{MainIntro} and give some observations about Conjecture \ref{HWC}; see Theorem \ref{Main} for a higher dimensional version of the next result.
 
\begin{thm} \label{MainIntro} Let $R$ be a one-dimensional Cohen-Macaulay local ring such that $R=S/(f)$ for some local ring $(S, \fn)$ and a non zero-divisor $f \in \fn$ on $S$. Let $M$ and $N$ be nonzero $R$-modules, and assume the following conditions hold:
\begin{enumerate} [\rm(i)]
\item $\pd_S(M)<\infty$, or $\pd_S(N)<\infty$ (e.g., $S$ is regular).
\item $\len(\Tor_i^R(M,N))<\infty$ for all $i\gg 0$ (e.g., $R$ is reduced).
\item $\theta^R(M,N)=0$, i.e., $\len(\Tor_{2n+2}^{R}(M,N))=\len(\Tor_{2n+1}^{R}(M,N))$ for some $n\gg 0$.
\end{enumerate}
If $M\otimes_{R}N$ is torsion-free, then $\Tor_i^R(M,N)=0$ for all $i\geq 1$, and $M$ and $N$ are torsion-free.
\end{thm}

The tool we employ to prove Theorem \ref{MainIntro} is the Hochster's $\theta$ invariant, which was initially defined by Hochster \cite{Ho1} to study the direct summand conjecture; it was further developed by Dao \cite{Da3, Da1}, and more recently by Buchweitz and Van Straten \cite{BIT}, and Walker et al. \cite{Mark3, Markarkiv}; see \ref{Dao} and \ref{Dao2}. To our best knowledge, Theorem \ref{MainIntro} is new, even if $S$ is a ramified regular ring; see \cite[3.6]{GORS2} and section 3.

Next is a corollary of Theorem \ref{MainIntro} concerning Conjecture \ref{HWC}; see Corollaries  \ref{extend}, \ref{corMF}.

\begin{cor} \label{sonucbir} Let $R$ be a one-dimensional Cohen-Macaulay ring with $R=S/(f)$ for some local ring $(S, \mathfrak{n})$ and a non zero-divisor $f \in \mathfrak{n}^2$ on $S$. Assume $M \in \md R$ is a module that has rank. If $M$ is not free, torsion-free, and $\pd_{S}(M)<\infty$, then $M\otimes_RM^{\ast}$ has torsion. In particular, if  $M=\coker(\alpha)$, where $(\alpha, \beta)$ is a reduced matrix factorization of $f$ over $S$ (i.e., a matrix factorization of $f$ with entries in $ \mathfrak{n}$), then $M\otimes_{R}M^{\ast}$ has torsion.
\end{cor}

As mentioned previously, if $S$ is regular, Corollary \ref{sonucbir} follows from a result of Huneke and Wiegand \cite[3.7]{HW1}. In this case, as well-known, maximal Cohen-Macaulay $R$-modules with no free summands occur as reduced matrix factorizations of $f$ over $S$; see \cite{Ei}. Similarly, if $S$ is \emph{G-regular} (i.e., when there are no non-free totally reflexive $S$-modules), Takahashi \cite{TakG} proved that there is a one-to-one correspondence between reduced matrix factorizations of $f$ and totally reflexive $R$-modules without free summands. Note that, if the ring $R$ is as in Corollary \ref{sonucbir}, reduced matrix factorizations of $f$ exist due to a result of Herzog, Ulrich and Backelin; see \cite[1.2 and 2.2]{HUB}, and also \cite[5.1.3]{Av2}, \cite[3.1]{AGP}, \cite[Chapter 8]{Yo}. 

In sections 2 and 3,  we collect some preliminary results and give a proof of Theorem \ref{MainIntro}, respectively. Section 4 is devoted to several applications of Theorem \ref{MainIntro} pertaining to torsion properties of tensor products of modules. As the gist of Theorem \ref{MainIntro} relies upon the vanishing of theta invariant, in section 5 we point out by an example that $\theta^R(M,N)$ can vanish non-trivially: in Example \ref{Omoshiroi}, we record an example of a one-dimensional reduced hypersurface ring $R$, and modules $M,N \in \md R$ such that $\theta^R(M,N)=0$, but neither $M$ nor $N$ has rank, or equivalently, neither $M$ nor $N$ has zero class in the reduced Grothendieck group $\overline{\G}(R)_{\QQ}$. Moreover, in section 6, building on an argument of Huneke and Wiegand \cite[4.7]{HW1}, we recall how to obtain examples of non-free, torsion-free modules $M$ with rank such that $M\otimes_RM$ is torsion-free over certain one-dimensional rings $R$; see \ref{Appendix2main}.

\section{Preliminaries}

In this section we collect some basic facts that will be used throughout the paper.

\begin{chunk} \textbf{Conventions.} Let $M\in \md R$ be a module. Then $\Omega M$ and $\Tr M$ denote the \emph{syzygy} and the \emph{transpose} of $M$, respectively.  For the definitions of standard homological invariants, such as the \emph{Gorenstein} dimension $\Gdim_R(M)$, \emph{complexity} $\cx_R(M)$ and the \emph{complete intersection} dimension $\CI_R(M)$ of $M$, we refer the reader to \cite{AuBr}, \cite{Av1} and  \cite{AGP}. 
\end{chunk}

\begin{chunk} \textbf{Torsion submodule.} \label{tors} Let $R$ be a local ring and let $M\in \md R$ be a module. The \emph{torsion submodule} $\tp R M$ of $M$ is the kernel of the natural map $M\to \text{Q}(R)\otimes_RM$, where $\text{Q}(R)$ is the total quotient ring of $R$. Hence there is an exact sequence of the form:
\begin{equation} \tag{\ref{tors}.1}
0\lra \tp RM \lra M\lra \tf RM\lra 0.
\end{equation}
$M$ is said to have \emph{torsion} (respectively,  \emph{torsion-free}) if $\tp RM\neq 0$ (respectively, $\tp RM= 0$). Note that, $M$ is torsion, i.e., $\tp RM=M$, if and only if $M_\fp = 0$ for each $\fp \in \Ass(R)$. 
\end{chunk}

\begin{chunk} \label{Yosh} Let $R$ be a local ring and let $L, M \in \md R$ be modules. If $M$ is maximal Cohen-Macaulay and $\pd_R(L)<\infty$, then $\Tor_{i}^R(L,M)=0$ for all $i\geq 1$; see, for example, \cite[2.2]{Yoshida}.
\end{chunk}


\begin{chunk} \label{CRS} (\cite[11.65]{Rotman}) Let $(S, \fn)$ be a local ring and let $R=S/(f)$ for some non zero-divisor $f\in \fn$ on $S$. If $M,N \in \md R$ are modules, then there is an exact sequence of the form:
\begin{equation}\notag{}
\Tor_2^S(M,N) \to \Tor_2^R(M,N) \to M\otimes_{R}N  \to  \Tor_1^S(M,N) \to \Tor_1^R(M,N)  \to 0. 
\end{equation}
\end{chunk} 


\begin{chunk} (\cite[3.1]{AM}) \label{kullan} Let $R$ be a local ring and let $N\in \md R$ such that $\Gdim_R(N)<\infty$. Then there is an exact sequence $0 \to L \to  Z \to N \to 0$,
where $\Gdim_R(Z)=0$ and $\pd_R(L)<\infty$. 
\end{chunk}

Next we collect certain properties of complexity and complete intersection dimension.

\begin{chunk} \label{AGProp}
Let $R$ be a local ring and let $0\neq M\in \md R$. Assume $\CI_R(M)<\infty$. Then,
\begin{enumerate}[\rm(i)]
\item $\cx_R(M) \leq \edim(R)-\depth(R)$; see \cite[5.6]{AGP}.
\item $\CI_{R_{\fp}}(M_{\fp})\leq \CI_R(M)$ for all $\fp\in \Spec(R)$; see \cite[1.6]{AGP} 
\item $\CI_R(M)=\sup\{i: \Ext_R^i(M,R)\neq 0\}=\depth(R)-\depth(M)$; see \cite[1.4]{AGP}
\item Let $f$ be a non zero-divisor on $R$. If $fM=0$, then $\CI_{R/fR}(M)<\infty$. Also, if $f$ is a non zero-divisor on $M$, then $\CI_{R/fR}(M/fM)=\CI_R(M)$; see \cite[1.12.2-3]{AGP}.
\item If $R \to R' $ is a flat local map of local rings and $\CI_{R'}(M\otimes_{R}R')<\infty$, then it follows $\CI_{R}(M)=\CI_{R'}(M\otimes_{R}R')$; see \cite[1.11]{AGP}. This fact does not require the finiteness of the complete intersection dimension $\CI_{R}(M)$ of $M$ over $R$. 
\item If $\CI_R(M)=0$, then it follows that $\CI(M^{\ast})=\CI_R(\Tr M)=0$, and also $\cx(M)=\cx(M^{\ast})<\infty$. Moreover, $\CI_R(M)=0$ if and only if $\CI_R(\Tr M)=0$; see \cite[3.5]{BJVH}, \cite[3.2]{BJcx} and \cite[3.2(i)]{Bounds}.
\item If $\Tor^{R}_{i}(M,N)=0$ for all $i\gg 0$, then $\Tor^R_{i}(M,N)=0$ for all $i\geq \CI_R(M)+1$; see \cite[4.9]{AvBu}. In particular, $\Tor_{i}^R(M,N)$ is torsion for all $i\gg 0$ if and only if $\Tor_{i}^R(M,N)$ is torsion for all $i\geq 1$; see \ref{tors}.
\item If $\Tor_{i}^R(M,N)=0$ for all $i\geq 1$, then it follows that the \emph{depth formula} for $M$ and $N$ holds, i.e., $\depth(M)+\depth(N)=\depth(R)+\depth(M\otimes_{R}N)$; see \cite[2.5]{ArY}.
\end{enumerate}
\end{chunk}

In the following we recall the definition of a version of Hochster's $\theta$ pairing \cite{Ho1}, developed by Dao in \cite{Da2}. This pairing can be defined in a more general setting, but the definition recorded here will suffice for our argument; see \cite{Da3, Da1} for more details.

\begin{chunk} \textbf{$\theta$ pairing.} \label{Dao} (\cite{Da2} and \cite{Ho1}) 
Let $M, N\in \md R$. Assume $R \to R' \twoheadleftarrow S$ is a codimension one quasi-deformation with zero-dimensional closed fibre, i.e., we have a diagram of local ring maps such that $R \to R'$ is flat, $R'\cong S/(f)$ for some non zero-divisor $f$ on $R'$, and $\dim(R'/\fm R')=0$. We set $(-)'=-\otimes_{R}R'$ and assume the following conditions hold:
\begin{enumerate}[\rm(a)]
\item $\CI_S(N')<\infty$ and $\Tor_i^S(M',N')=0$ for all $i\gg 0$ (e.g., $\pd_S(N')<\infty$).
\item $\len(\Tor_i^R(M,N))<\infty$ for all $i\gg 0$ (e.g., $R$ is an isolated singularity).
\end{enumerate} 

It follows that $\CI_{R'}(N')<\infty$ and $\CI_R(N)=\CI_{R'}(N')$; see \ref{AGProp}(iv, v). Note we have, by (a) and \ref{AGProp}(vii), that $\Tor_i^S(M',N')=0$ for all $i>\CI_S(N')$. Therefore \cite[11.65]{Rotman} yields the following isomorphisms:
\begin{equation}\tag{\ref{Dao}.1}
\Tor_{i}^{R'}(M',N') \cong \Tor_{i+2}^{R'}(M',N') \text{ for all } i>\CI_{R'}(N').
\end{equation}

For a non-maximal prime ideal $\fp$, we have $\Tor_i^R(M,N)_{\fp}=0$ for all $i>\CI_R(N)$; see \ref{AGProp}(ii, vii). Thus $\len(\Tor_i^R(M,N))<\infty$ for all $i>\CI_R(N)$, and hence
\begin{equation}\tag{\ref{Dao}.2}
\len_{R'}(\Tor_i^{R'}(M',N'))<\infty \text{ for all } i>\CI_{R'}(N').
\end{equation}
Let $\ell=\len_{R'}(R'/\fm R')$. Then, by (\ref{Dao}.1) and (\ref{Dao}.2), we see that the difference
\begin{align}\notag{}
& \len(\Tor_{2n+2}^{R}(M,N))-\len(\Tor_{2n+1}^{R}(M,N))=& \\& \frac{1}{\ell} \cdot \left( \len_{R'}(\Tor_{2n+2}^{R'}(M',N'))-\len_{R'}(\Tor_{2n+1}^{R'}(M',N')) \right) \notag{}
\end{align}
is independent of $n$ if $2n>\CI_{R'}(N')-1$. One defines the theta pairing over $R$ as:
\begin{equation}\notag{}
\theta^R(M,N)=\len(\Tor_{2n+2}^{R}(M,N))-\len(\Tor_{2n+1}^{R}(M,N)),
\end{equation}
where $n$ is an integer with $2n>\CI_R(N)-1$.

It follows from the definition that $\theta^R$ is additive on short exact sequence of modules in $\md R$, whenever it is well-defined on each pair of modules in question.

\end{chunk}

\begin{chunk} \label{Dao2} 
Let $M, N\in \md R$. Assume the following conditions hold:
\begin{enumerate}[\rm(a)]
\item $\CI_R(N)<\infty$ and $\cx_{R}(N)=1$. 
\item $\len(\Tor_i^R(M,N))<\infty$ for all $i\gg 0$
\end{enumerate} 
Then we can choose a codimension one quasi-deformation of the form $R \to R' \stackrel{\alpha}{\twoheadleftarrow} S$, where $\pd_{S}(N')<\infty$; see \cite[4.1.3]{AvBu}. Localizing at some $\fp \in \Min_{R'}(R'/ \fm R')$ with $\fq=\alpha^{-1}(\fp)$, we see that $R \to R'_{\fp} \stackrel{\alpha}{\twoheadleftarrow} S_{\mathfrak{\fq}}$ is a codimension one quasi-deformation with $\pd_{S_{\fq}}(N\otimes_{R}R'_{\fp})<\infty$; see the proof of \cite[2.11]{Sean}. Therefore, replacing the original quasi-deformation with the aforementioned one, we may assume $\dim(R'/\fm R')=0$. 

So it follows from \ref{Dao} that $\theta^R(M,N)=\len(\Tor_{2n+2}^{R}(M,N))-\len(\Tor_{2n+1}^{R}(M,N))$  is well-defined, as long as $n$ is an integer with $2n>\CI_R(N)-1$.
\end{chunk}




\section{Proof of the main result}

In this section we prove the main result of this paper; see Theorem \ref{Main}. Our motivation comes from the following result, which is recorded for the one-dimensional case:

\begin{chunk} (Celikbas, Piepmeyer, Iyengar and Wiegand; see \cite[3.6]{GORS2}) \label{thm:eta} 
Let $R$ be a one-dimensional local ring with $\widehat{R}=S/(f)$ for some unramified regular local ring $S$, and a non zero-divisor $f\in \fn$ on $S$. Let $M, N \in \md R$ be nonzero modules. Assume $\len(\Tor_i^R(M,N))<\infty$ for all $i\gg 0$, and $\theta^R(M,N)=0$. If $M\otimes_{R}N$ is torsion-free, then $\Tor^{R}_{i}(M,N)=0$ for all $i\ge 1$, and $M$ and $N$ are both torsion-free. \pushQED{\qed} 
\qedhere 
\popQED	
\end{chunk}

A consequence of our argument gives an extension of \ref{thm:eta} and establishes the vanishing of $\Tor^{R}_{i}(M,N)$ when $S$ is an arbitrary two-dimensional Cohen-Macaulay local ring, and $\pd_S(N)<\infty$ or $\pd_S(M)<\infty$; see Theorem \ref{Main}. As is clear, since we do not work over hypersurface rings, our method of proof is different from that employed to prove \ref{thm:eta}. Among other things, one of the properties that is not available to us under our setup is that, when $R$ is a hypersurface, every torsion-free module can be embedded in a free $R$-module; see \cite[1.5]{HW1}. Also, over a ring $R$ as in \ref{thm:eta}, for a pair of modules $(M,N)$ in $\md R$, if $\theta^R(M,N)$ is defined and vanishes, then the pair $(M,N)$ is Tor-rigid \cite[2.8]{Da1}; this Tor-rigidity result depends on the fact that $S$ is an unramified regular ring. Thus the properties that play an important role in the proof of \ref{thm:eta} do not apply directly under our setup.

The following is our main result; although we are mainly interested in the one dimensional case
(due to Conjecture \ref{HWC}), our argument works over Cohen-Macaulay local rings of arbitrary positive dimension as long as modules considered have sufficiently large depth.

\begin{thm} \label{Main} Let $R$ be a Cohen-Macaulay local ring, and let $M, N \in \md R$ be modules. Assume $ \dim(R)=d\geq 1$, and the following conditions hold:
\begin{enumerate} [\rm(i)]
\item $\CI_R(N)<\infty$ and $\cx_{R}(N)\leq 1$.
\item $\len(\Tor_i^R(M,N))<\infty$ for all $i\gg 0$.
\item $\depth(M)\geq d-1$ and $\depth(N)\geq d-1$.
\item If $d=1$, assume further $\theta^R(M,N)=0$.
\end{enumerate}
If $M\otimes_{R}N$ is (nonzero) maximal Cohen-Macaulay, then $\Tor_i^R(M,N)=0$ for all $i\geq 1$, and $M$ and $N$ are both maximal Cohen-Macaulay.
\end{thm}

\begin{proof} It suffices to prove the vanishing of $\Tor_i^R(M,N)$ for all $i\geq 1$; see \ref{AGProp}(viii). 

We first assume $d\geq 2$, and choose a non zero-divisor $x$ on $R$, $M$, $N$ and $M\otimes_{R}N$. Setting $T=R/xR$, $A=M/xM$ and $B=N/xN$, we can see that $\theta^T(A,B)$ is defined and vanishes. Also, if $\Tor_i^T(A,B)=0$ for all $i\geq 1$, then we can show $\Tor_i^R(M,N)=0$ for all $i\geq 1$. Since $\depth_{T}(A)\geq \depth(T)-1$, $\depth_{T}(B)\geq \depth(T)-1$, $A\otimes_{T}B$ is maximal Cohen-Macaulay over $T$, $\CI_{T}(B)<\infty$ and $\cx_{T}(B)=\cx_{R}(N)$, it suffices to replace the pair $(M,N)$ over the ring $R$ with the pair $(A,B)$ over the ring $T$, and consider the theorem for the case where $d=1$.

We proceed by assuming $d=1$. If $\pd_R(N)<\infty$, then $\Tor^R_i(\tf RM, N)=0$ for all $i\geq 1$, which implies $\Tor_i^R(M,N)=0$ for all $i\geq 1$; see \ref{Yosh}. So  we may assume $\cx_{R}(N)=1$.
Now choose a quasi-deformation $R \to R' \twoheadleftarrow S$ such that $\dim(R'/\fm R')=0$, $R'=S/(f)$ for some local ring $(S,\fn)$ and a non zero-divisor $f\in \fn$ on $S$ with $\pd_{S}(N\otimes_{R}R')<\infty$; see \ref{Dao2}. Therefore we may assume $R=R'=S/(f)$ with $\pd_{S}(N)<\infty$.

Set $U=\tf RM$ and $V=\tp RM$. As $\dim(R)=1$, we know $\len(V)<\infty$. Thus $\theta^R(V,N)$ is well-defined; see (\ref{Dao2}). Also the short exact sequence $0 \to V \to M \to U \to 0$ implies 
$\len(\Tor_i^R(U,N))<\infty$ for all $i\gg 0$. In particular, $\theta^R(U,N)$ is well-defined and hence:
\begin{equation} \tag{\ref{Main}.4}
0=\theta^{R}(M, N)=\theta^{R}(U, N)+\theta^R(V, N).
\end{equation}

\noindent \emph{Claim}. $\theta^{R}(V, N)=0$.\\
\emph{Proof of the claim}. To prove the claim, we follow the argument of \cite[4.6]{CeD2}. Note that $\pd_S(\Omega_R N)=1$. It now follows from \ref{CRS} that there is an exact sequence:
\begin{equation} \tag{\ref{Main}.5}
0 \to \Tor_2^R(\Omega_R N,k) \to 
\Omega_R N\otimes_{S}k \to \Tor_1^S(\Omega_R N,k) \to  \Tor_1^R(\Omega_R N,k) \to 0.
\end{equation}
Taking the alternating sum of lengths of modules in (\ref{Main}.5), we obtain:
\begin{equation}  \notag{}
\theta^R(k, \Omega_R N)=\beta_2^R(\Omega_R N)-\beta_1^R(\Omega_R N)=\beta_0^S(\Omega_R N)-\beta_1^S(\Omega_R N),
\end{equation}
where $\beta(\Omega_R N)$ denotes the Betti number of $\Omega_R N$. The Euler characteristic of $\Omega_R N$ over $S$, which is $\beta_0^S(\Omega_R N)-\beta_1^S(\Omega_R N)$, vanishes since $f\cdot \Omega_R N=0$; see \cite[19.8]{Mat}. Therefore we have $\theta^{R}(k, \Omega_R N)=0$. As $\theta^{R}(k, N)=-\theta^{R}(k, \Omega_R N)$, we see $\theta^{R}(k, N)=0$. Moreover, since $V$ has a finite filtration by copies of $k$, it follows that $\theta^{R}(V, N)$ vanishes. This justifies the claim.
	
Now, by (\ref{Main}.4) and the claim, we have $\theta^{R}(U, N)=0$. We proceed by considering the short exact sequence that follows from \ref{kullan}:
\begin{equation} \tag{\ref{Main}.6}
0 \to L \to  Z \to N \to 0,
\end{equation}
Here $L$ is free, $\CI_R(Z)=0$ and $\pd_S(Z)=1$. Tensoring (\ref{Main}.6) with $U$ over $R$, we see that $\Tor_i^R(U,Z) \cong \Tor_{i}^R(U, N)$ for each $i\geq 2$, and obtain the following exact sequence:
\begin{equation}\tag{\ref{Main}.7}
  0 \to \Tor_1^R(U,Z) \to \Tor_{1}^R(U,N) \to  U\otimes_{R}L \to U\otimes_{R}Z \to U\otimes_{R}N \to 0.
\end{equation}
Note that, for each $i\geq 1$, $\Tor_{i}^R(U,N)$ is torsion and $\len(\Tor_i^R(U,Z))<\infty$; see \ref{AGProp}(ii, vii). Hence, since $U\otimes_RL$ is torsion-free, it follows from (\ref{Main}.7) that $\Tor_i^R(U,Z) \cong \Tor_{i}^R(U,N)$ for each $i\geq 1$, and $U\otimes_RZ$ is torsion-free. Once again we consider the exact sequence from \ref{CRS}, this time for the pair $(U,Z)$:
\begin{equation}\notag{}
\Tor_2^S(U,Z) \to \Tor_2^R(U,Z) \to U\otimes_RZ  \to  \Tor_1^S(U,Z) \to \Tor_1^R(U,Z)  \to 0. 
\end{equation}
Since $\Tor_2^S(U,Z)=0$, $\depth_R(U\otimes_{R}Z)=1$ and $\len(\Tor_2^R(U,Z))<\infty$, we conclude that $\Tor_2^R(U,Z)$ vanishes. Moreover, we have:
\begin{equation}\notag{}
0=\theta^R(U,N)=\theta^R(U,Z)=\len( \Tor_{2n+2}^{R}(U,Z) )- \len( \Tor_{2n+1}^{R}(U,Z)) \text{ for } n\geq 0.
\end{equation}
As $\Tor_2^R(U,Z)=0$, we see that $\Tor_i^R(U,Z)=0$ for all $i\geq 1$. This implies the vanishing of $\Tor_i^R(U,N)$, as well as the vanishing of $\Tor_i^R(M,N)$, for each $i\geq 1$.
\end{proof}

We finish this section by noting a related result: in case $M$ is maximal Cohen-Macaulay, one can prove the following, which has no depth assumption on $N$. 

\begin{prop} \label{bitti} Let $R$ be a Cohen-Macaulay local ring of dimension $d\geq 1$ such that $R=S/(f)$ for some local ring $S$, and a non zero-divisor $f\in \fn$ on $S$. Let $M\in \md R$ be maximal Cohen-Macaulay, and let $N\in \md R$ be a module. Assume the following hold:
\begin{enumerate} [\rm(i)]
\item $\CI_S(N)<\infty$ and $\Tor_i^S(M,N)=0$ for all $i\gg 0$ (e.g., $\pd_S(N)<\infty$).
\item $\len(\Tor_i^R(M,N))<\infty$ for all $i\gg 0$. 
\item $\theta^R(M,N)=0$.
\end{enumerate}
If $M\otimes_{R}N$ is torsion-free, then $\Tor_i^R(M,N)=0$ for all $i\geq 1$.
\end{prop}

We skip giving a proof of Proposition \ref{bitti} since it can be established by using the ideas employed in the proof of Theorem \ref{Main}.

\section{Some corollaries of the main result}

In this section we proceed to give various corollaries of Theorem \ref{Main} concerning the torsion submodule of tensor products of modules, especially those of the form $M\otimes_RM^{\ast}$ over one dimensional local rings. In particular we give a proof of Corollary \ref{sonucbir}; see Corollaries \ref{extend} and \ref{corMF}. Along the way we extend results of Huneke and Wiegand \cite{HW1}, and Auslander \cite{Au} on the reflexivity of tensor products of modules which justify Conjecture \ref{HWC} over normal domains; see Proposition \ref{Augen}.

We denote by $\G(R)$ the \textit{Grothendieck group} of modules in $\md R$, i.e., the quotient of the free abelian group of all isomorphism classes of modules in $\md R$ by the subgroup generated by the relations coming from short exact sequences of modules in $\md R$. We write $[M]$ for the class of $M$ in $\G(R)$ and denote by $\overline{\G}(R)$ the group $\G(R)/\ZZ \cdot [R]$, the reduced Grothendieck group of $R$. We set $\overline{\G}(R)_{\QQ}=(\G(R)/\ZZ \cdot [R])\otimes_{\ZZ}\QQ$. 



The next corollary corroborates \cite[1.2]{Ce}, which examines the vanishing of Tor for modules of complexity at most one over complete intersection rings.

\begin{cor} \label{coryeni} Let $R$ be a one-dimensional local ring and let $M, N\in \md R$ be nonzero modules. Assume the following hold:
\begin{enumerate} [\rm(i)]
\item $\CI_R(N)<\infty$ and $\cx_{R}(N)\leq 1$.
\item $N$ is locally free on $\Ass(R)$.
\item $[M]=0$ in $\overline{\G}(R)_{\QQ}$.
\end{enumerate}
If $M\otimes_{R}N$ is torsion-free, then $\Tor_i^R(M,N)=0$ for all $i\geq 1$, and $M$ and $N$ are torsion-free. Moreover, the pair $(M,N) $ is Tor-rigid.
\end{cor}

\begin{proof} We may assume $R$ is Cohen-Macaulay; otherwise $N$ would be free and the claims follow. Notice $\theta^R(N,-):\overline{G}(R)_{\QQ} \to \QQ$ is a well-defined function since $N$ is locally free on $\Ass(R)$; see \ref{Dao2}. As $[M]=0$ in $\overline{G}(R)_{\QQ}$, we have that $\theta^R(M,N)=0$ so the result follows from Theorem \ref{Main}.

To prove the claim on Tor-rigidity, we use the exact sequence that follows from \ref{kullan}:
\begin{equation} \tag{\ref{coryeni}.1}
0 \to F \to  Z \to N \to 0,
\end{equation}
Here $F$ is free, $\CI_R(Z)=0$ and $\cx_R(Z)=\cx_R(N)$. Now assume $\Tor_1^R(M,N)=0$. Applying $-\otimes_{R}Z$ to the syzygy exact sequence $0\to \Omega M \to R^{\oplus v} \to M \to 0$, we see $\Omega M \otimes_{R}Z$ is torsion-free. Hence it follows from the previous part that $\Tor_i^R(\Omega M,Z)=0$ for all $i\geq 1$. This implies $\Tor_i^R(M,N)=0$ for all $i\geq 1$. 
\end{proof}

\begin{rmk} \label{rmkyeni} If $R$ is a one-dimensional local ring and let $M\in \md R$ is a module (not necessarily torsion-free) which has rank, then it follows $[M]=0$ in $\overline{\G}(R)_{\QQ}$; see \cite[2.5]{CeD} and \cite[1.3]{HW1}. Therefore it follows from Corollary \ref{coryeni} that, if $R$ is a one-dimensional local ring and $M,N\in \md R$ are nonzero modules such that $\CI_R(N)<\infty$, $\cx_{R}(N)\leq 1$, $N$ is locally free on $\Ass(R)$, $M$ has rank and $M\otimes_{R}N$ is torsion-free, then $\Tor_{i}^R(M,N)=0$ for all $i\geq 1$, and $M$ and $N$ are torsion-free. \pushQED{\qed} 
\qedhere
\popQED	
\end{rmk}

\begin{rmk} \label{rmktorsion} \pushQED{\qed} 
Let $R$ be a one-dimensional domain and $M, N \in \md R$ be nonzero modules. 

If $M\otimes_{R}N$ is torsion-free, then $\Ext_R^1(\Tr M, N)=0$. Therefore, if $M\otimes_{R}N$ is torsion-free and $N$ is Tor-rigid, it follows that $\Ext_R^1(\Tr M, R)=0$, i.e., $M$ is torsion-free; see \cite{Au}. 

If $R$ is a complete intersection, it is an open problem whether $N$ must be Tor-rigid, and whether the torsion-freeness of $M\otimes_{R}N$ implies the torsion-freeness of $M$; see \cite[2.1]{CeRo}. Corollary \ref{coryeni} gives a partial affirmative answer and shows that both $M$ and $N$ are torsion-free in case $M\otimes_{R}N$ is torsion-free, and $M$ or $N$ has complexity at most one, i.e., has bounded Betti numbers. 
\qedhere
\popQED	
\end{rmk}

It is known that the conclusion of Corollary \ref{coryeni} may fail in case $[M]\neq 0$ in $\overline{\G}(R)_{\QQ}$. For example, if $R=k[\![x,y]]\!/(xy)$, $M=R/(x)$ and $N=R/(x^2)$, then $M$ and $N$ are both locally free on $\Ass(R)$, but $\Tor_s^R(M,N)=0$ if and only if $s$ is either a negative integer or a positive even integer; see \cite[page 164]{HW2} and also \ref{Giro}. Indeed such a vanishing result occurs in general; in passing we record this fact as a proposition.


\begin{prop} \label{propAA} Let $R$ be a one-dimensional local ring and let $M,N\in \md R$ be modules, both of which are locally free on $\Ass(R)$. Assume $\CI_R(N)<\infty$ and $\cx_{R}(N)\leq 1$. 
If $M\otimes_{R}N$ is torsion-free, then $\Tor_{2i}^R(\tf RM,N)=0$ for all $i\geq 1$. 
\end{prop}

\begin{proof} We may replace $M$ with $\tf RM$, and assume $M$ is torsion-free. 
There is a short exact sequence of the form
\begin{equation} \tag{\ref{propAA}.1}
0 \to F \to  Z \to N \to 0,
\end{equation}
where $F$ is free, $\CI_R(Z)=0$ and $\cx_R(Z)=\cx_R(N)$; see \ref{kullan}.
It follows from \cite[3.1]{BJVH} that there exists a flat local map $R\to R'$ and an exact sequence in $\md R'$ of the form
\begin{equation}\tag{\ref{propAA}.2}
0 \to Z' \to K \to \Omega_{R'}(Z') \to 0,
\end{equation}
where $Z'=Z\otimes_RR'$, $\depth_{R'}(Z')=\depth_{R'}(K)$ and $\pd_{R'}(K)<\infty$ (note $\Omega_{R'}(Z')=\Omega_R(Z)'$). Since $\depth(R')-\depth_{R'}(Z')=\depth(R)-\depth_R(Z)=\CI_R(Z)=0$, we see $K$ is a free $R'$-module. Hence (\ref{propAA}.2) yields an injection $\Tor_1^{R'}(\Omega_{R}^1(Z)', M') \hookrightarrow M'\otimes_{R'}Z'$. 

Tensoring (\ref{propAA}.1) with $M$ over $R$, we see that $M\otimes_R Z$ is torsion-free. Moreover, as $M\otimes_R Z$ is locally free on $\Ass(R)$, it embeds into a free $R$-module. Therefore $M'\otimes_{R'}Z'$ torsion-free over $R'$. Hence $\Tor_1^{R'}(\Omega_{R}(Z)', M')$, being a torsion $R'$-module, vanishes. This implies $\Tor_{2i}^R(M,N)=0$ for all $i\geq 1$; see (\ref{propAA}.1).
\end{proof}


Our next observation maybe of independent interest.

\begin{lem} \label{lemiyi} Let $R$ be a one-dimensional local ring and let $M\in \md R$.
\begin{enumerate}[\rm(a)]
\item Assume $M$ is torsion-free. Then $\CI_R(M)<\infty$ if and only if $\CI_R(M^{\ast})<\infty$.
\item If $\CI_R(M)<\infty$, $\Tor_{i}^R(M, M^{\ast})=0$ for all $i\gg 0$, and $M\otimes_{R}M^{\ast}$ is a nonzero torsion-free module, then $M$ is free.
\end{enumerate}
\end{lem}

\begin{proof} (a) It suffices to assume $\CI(M^{\ast})<\infty$ and prove $\CI_R(\Tr M)=0$; see \ref{AGProp}(vi). Hence we assume $\CI_R(M^{\ast})<\infty$ so that $\CI_R(\Tr M)<\infty$; see \cite[3.6]{Sean}.

Let $p\in \Ass(R)$. Then, since $\CI_{R_{\fp}}(M_{\fp})=0$, it follows $M_{\fp}$ is totally reflexive over $R_{\p}$. Therefore $\Ext^1_R(\Tr M, R)=0$ so that $\CI_R(\Tr M)=0$; see \ref{AGProp}(iii).

(b) Note that each $\Tor^R_i(M,M^{\ast})$ has finite length for $i\geq 1$; see \ref{AGProp}(vii). Hence it follows from \cite[3.6]{Bounds} that $\Tor_{i}^R(M, M^{\ast})=0$ for all $i\geq 1$. Since $M\otimes_{R}M^{\ast}$ is nonzero and torsion-free, the depth formula implies $\CI_R(M)=0$; see \ref{AGProp}(viii). Therefore, since $\Tor_{i}^R(M, \Tr M)$ for all $i\gg 0$ and $\CI_R(M)=0$, we conclude that $\Tor_{1}^R(M, \Tr M)=0$; see \ref{AGProp}(vii). Thus $M$ is free; see, for example, \cite[3.9]{Yo}. 
\end{proof}

If $R=S/(f)$, where $(S, \mathfrak{n})$ is a two-dimensional regular local ring and $0\neq f \in \mathfrak{n}$, it follows from a result of Huneke and Wiegand \cite[3.7]{HW1} that $M\otimes_RM^{\ast}$ has torsion for each non-free, torsion-free module $M\in \md R$ with rank.  Hence Conjecture \ref{HWC} is true  over hypersurface rings. In the following, we will generalize this fact and show that it carries over non-hypersurface rings under mild conditions; see Corollaries \ref{extend}, \ref{corMF} and \ref{extend2}.

\begin{cor} \label{extend} Let $R$ be a one-dimensional Cohen-Macaulay ring such that $R=S/(f)$ for some local ring $(S, \mathfrak{n})$ and $f \in \mathfrak{n}$ is a non zero-divisor on $S$. Let $M \in \md R$ be a module such that $M\otimes_RM^{\ast}$ is torsion-free, and $\pd_{S}(M)<\infty$ or $\pd_{S}(M^{\ast})<\infty$. 

If $M$ is torsion-free, $\len(\Tor_i^R(M,M^{\ast}))<\infty$ for all $i\gg 0$, and  $\theta^R(M,M^{\ast})=0$, then $M$ is free. In particular, if $M$ has rank and $M\otimes_RM^{\ast}$ is nonzero (e.g., $M$ is a nonzero torsion-free module with rank), then $M$ is free.
\end{cor}

\begin{proof} Assume $M$ is torsion-free. Then, by Lemma \ref{lemiyi}(a), we have $\CI_R(M)=0$. So $\cx(M)=\cx(M^{\ast})\leq 1$; see \ref{AGProp}(vi). Therefore Theorem \ref{Main} yields the vanishing of $\Tor_{i}^R(M, M^{\ast})$ for all $i\geq 1$. Now $M$ must be free by Lemma \ref{lemiyi}(b).

For the case where $M$ has rank and $M\otimes_RM^{\ast}$ is nonzero, it follows from Remark \ref{rmkyeni} that $\Tor_i^R(M,M^{\ast})=0$ for all $i\geq 1$ so that $M$ is torsion-free; see \ref{AGProp}(viii). This implies $M$ is free, for example, by the previous part.
\end{proof}

\begin{rmk} We note that, one can use  \cite[4.2 and 4.4.7]{AvBu} and \cite[4.6]{HW1}, and give a proof to the first part of Corollary \ref{extend} for Gorenstein rings without appealing to Theorem \ref{Main}.
\pushQED{\qed} 
\qedhere 
\popQED	
\end{rmk}


If $(S, \mathfrak{n})$ is a Cohen-Macaulay local ring and $f\in \mathfrak{n}^2$ is a non zero-divisor on $S$, then $f$ has a reduced matrix factorization $(\varphi, \psi)$ over $S$. In this case, $\coker (\varphi)$ is a non-free, maximal Cohen-Macaulay module over $S/(f)$ with projective dimension one over $S$; see \cite{HUB}.

Recall that a local ring $S$ is called \emph{G-regular} \cite{TakG} if each totally reflexive module in $\md S$ is free. Note that each regular ring, as well as each Golod ring, is G-regular.

The following, advertised in Corollary \ref{sonucbir}, follows from Corollary \ref{extend} and \cite[2.13]{TakG}.

\begin{cor} \label{corMF} Let $R=S/(f)$ be a one-dimensional Cohen-Macaulay ring, where $(S, \mathfrak{n})$ is a local ring and $f \in \mathfrak{n}^2$ is a non zero-divisor on $S$. If $(\varphi, \psi)$ is a reduced matrix factorization of $f$, and $M=\coker (\varphi)$ has rank, then $M\otimes_{R}M^{\ast}$ has torsion. In particular, if $S$ is G-regular and $M \in \md R$ is a non-free, totally reflexive $R$-module which has rank, then $M\otimes_{R}M^{\ast}$ has torsion. 
\end{cor}

Here is an example for which we can employ Corollary \ref{corMF}; note the ring in question is a complete intersection, but not a hypersurface; see also \cite[4.17]{Ce} and cf. \cite[3.7]{HW1}.

\begin{eg} \label{egMF} Let $R=S/(f)$, where $S=\CC[\![x,y,z]\!]/(xz-y^2)$ and $f=x^3-z^2$. Then 
it follows that $R=\CC[\![t^4, t^5, t^6]\!]$ is a one-dimensional domain. Moreover,
$$(\varphi, \psi)=
\left(
\begin{pmatrix}
-z & x \\
x^2 & -z 
\end{pmatrix}
, 
\begin{pmatrix}
z & x \\
x^2 & z 
\end{pmatrix}
\right) 
$$
is a reduced matrix factorization of $f$ over $S$. \pushQED{\qed} 
So, by Corollary \ref{corMF}, $M\otimes_{R}M^{\ast}$ has torsion, where $M \in \md R$ is the module given by the exact sequence $S^{\oplus 2} \stackrel{\varphi}{\rightarrow}S^{\oplus 2} \to M \to 0$. \qedhere
\popQED	
\end{eg}




Next is another result that follows from Corollary \ref{extend}; it is an extension of the result of Huneke and Wiegand mentioned preceeding Corollary \ref{extend}. It also extends \cite[4.17]{Ce}, which is limited to complete intersection rings.

\begin{cor} \label{extend2} Let $R$ be a one-dimensional local ring and let $M \in \md R$ be a non-free module. Assume $M$ has rank (e.g., $R$ is a domain). Assume further $\CI_R(M)<\infty$ and $\cx_R(M)\leq 1$ (e.g., $R$ is a hypersurface). If $M\otimes_RM^{\ast}$ is not zero, then it has torsion. In particular, if $M$ is torsion-free, then $M\otimes_RM^{\ast}$ has torsion.
\end{cor}


Under a similar setting, the conclusion of Corollary \ref{extend2} still holds in case $M$ does not have rank, but the length of certain Tor modules coincide; this fact relies upon Corollary \ref{extend} and a result of Bergh \cite{BVH}. We can see this as follows:

\begin{cor} \label{t2} Let $R$ be a one-dimensional Cohen-Macaulay ring such that $R=S/(f)$ for some local ring $(S, \mathfrak{n})$ and $f \in \mathfrak{n}$ is a non zero-divisor on $S$. Let $M \in \md R$ be a non-free, torsion-free module such that either $\pd_{S}(M)<\infty$ or $\pd_{S}(M^{\ast})<\infty$. Assume $\len(\Tor_n^R(M,M^{\ast}))=\len(\Tor_{n+q}^R(M,M^{\ast}))<\infty$ for some even integer $n\geq 1$ and an odd integer $q\geq 1$. Then $M\otimes_{R}M^{\ast}$ has torsion.
\end{cor}

\begin{proof} Notice $\CI_R(M)=0$ and $1=\cx_R(M)=\cx_R(M^{\ast})$; see Lemma \ref{lemiyi}(a) and \ref{AGProp}(vi). Let $\fp\in \Supp_R(M)$ be a nonmaximal prime ideal of $R$. Then it follows $\CI_{R_{\fp}}(M_{\fp})=0$ and $\Tor_n^R(M,M^{\ast})_{\fp}=\Tor_{n+q}^R(M,M^{\ast})_{\fp}=0$; see \ref{AGProp}(ii). Therefore, by \cite[3.2]{BVH}, we have $\len(\Tor_{i}^R(M,M^{\ast})<\infty$ for each $i\geq 1$. Also, since $\Tor_i^R(M,M^{\ast})\cong  \Tor_{i+2}^R(M,M^{\ast})$ for each $i\geq 1$, we see that $\len(\Tor_i^R(M,M^{\ast}))=\len(\Tor_{i+1}^R(M,M^{\ast}))$ for each $i\geq 1$. In particular $\theta^R(M,M^{\ast})=0$, so the result follows from Corollary \ref{extend}.
\end{proof}

\subsection*{Further remarks related to Conjecture \ref{HWC}} 

Huneke and Wiegand \cite[5.2]{HW1} proved that, if $R$ is a local domain satisfying Serre's condition $(S_2)$, and $R_{\fp}$ is a hypersurface for each height-one prime ideal $\fp$ of $R$, then $M\otimes_RM^{\ast}$ is not reflexive for each non-free, torsion-free module $M\in \md R$. This result was motivated by a theorem of Auslander \cite[3.3]{Au} which justifies Conjecture \ref{HWC} over normal domains: if $R$ is a local normal domain and $M\in \md R$ is a non-free, torsion-free $R$-module, then $M\otimes_{R}M^{\ast}$ is not reflexive. We will see in Proposition \ref{Augen} that both of these results hold more generally. First we need:

\begin{lem} \label{lemHWrmk} Let $R$ be a local ring, and let $M\in \md R$ be a module such that $M^{\ast}\neq 0$. If $\Ext^1_R(\Tr M, M^{\ast})=\Ext^2_R(\Tr M, M^{\ast})=0$, then $M$ is free.
\end{lem}

\begin{proof} There is an exact sequence $0 \to M^{\ast} \to F \to G \to \Tr M \to 0$, where $F, G \in \md R$ are free modules. This yields the following short exact sequences:  
\begin{equation}\tag{\ref{lemHWrmk}.1}
0 \to M^{\ast} \to F \to L \to 0 \;\; \text{  and  } \;\; 0 \to L \to G \to \Tr M \to 0.
\end{equation}
Note $\Ext^1_R(L,M^{\ast})=0$ so $M^{\ast}$ is free. Hence we have $\Ext^1_R(\Tr M, R)=\Ext^2_R(\Tr M, R)=0$, i.e., $M$ is reflexive. This implies $M$ is free.
\end{proof}
\begin{prop} \label{Augen} Let $R$ be a local ring satisfying Serre's condition $(S_2)$ and let $M\in \md R$ be a module such that $M^{\ast}\neq 0$ and $M\otimes_RM^{\ast}$ is reflexive. Then $M$ is free if one of the following conditions holds:
\begin{enumerate}[\rm(i)]
\item $M_{\fp}$ is free over $R_{\fp}$ for all $\fp \in \Spec(R)$ with $\dim(R_{\fp})\leq 1$.
\item $M$ has rank, and $R_{\fp}$ is a hypersurface for all $\fp \in \Spec(R)$ with $\dim(R_{\fp})=1$.
\end{enumerate}
\end{prop}

\begin{proof} For part (i), one can show that $\Ext^1_R(\Tr M, M^{\ast})=0=\Ext^2_R(\Tr M, M^{\ast})$ by proceeding as in the proof of  \cite [3.3]{Au}. Hence the claim follows from Lemma \ref{lemHWrmk}.

For part (ii), notice, since $M^{\ast}\neq 0$ and the rank of $M$ is positive, it follows $M_{\fq}\neq 0 \neq M^{\ast}_{\fq}$ for all $\fq \in \Spec(R)$. Now let $\fp \in \Spec(R)$ with $\dim(R_{\fp})=1$. Then $M_{\fp} \otimes_{R_{\fp}} M_{\fp}^{\ast}\neq 0$ so that $M_{\fp}$ is free over $R_{\fp}$; see Corollary \ref{extend2}. Now the result follows from part (i).
\end{proof}



We finish this section by recording a few observations about \emph{Ulrich} ideals related to our argument. We refer the reader to \cite{GotoUlrich} for the definition and basic properties of Ulrich ideals. For our purpose, we note:

\begin{chunk} \label{Ulrichrmk} If $R$ is a Gorenstein ring and $I$ is an Ulrich ideal of $R$, then  $\cx_R(I)\leq 1$; see \cite[7.4]{GotoUlrich}.
\end{chunk}

\begin{cor} \label{Ulrich} Let $R$ be a one-dimensional complete intersection domain and let $I$ be an Ulrich ideal of $R$. Then $R/I$ is Tor-rigid. Moreover, if $M \in \md R$ is a module that has torsion, then $M\otimes_RI$ has torsion.
\end{cor}

\begin{proof} This claim follows from Corollary \ref{coryeni} and \ref{Ulrichrmk}.
\end{proof}

\begin{eg} \label{ExUlrich} Let $R = \CC[\![t^4, t^5, t^6]\!]=\CC[\![x,y,z]\!]/(xz-y^2, x^3-z^2)$ and $I=(t^4, t^6)$. Then $R$ is a one-dimensional complete intersection domain, and $I$ is an Ulrich ideal of $R$; see \cite[6.3]{GotoUlrich}. Hence, $R/I$ is Tor-rigid, and $I\otimes_{R}M$ has torsion for each module $M\in \md R$ that has torsion; see Corollary \ref{Ulrich}. Also, letting $J=(t^4, t^5)$, we have that $I\otimes_{R}J$ is torsion-free, i.e., $\Tor_2^R(R/I, R/J)=0$; see \cite[4.3]{HW1}. So we conclude that $\Tor_i^R(I,J)=0$ for all $i\geq 1$.
\end{eg}

Notice, \ref{Ulrichrmk}, in conjuction with Corollary \ref{extend2}, establishes Example \ref{propA}(ii) over complete intersection rings. In fact this result is true over Gorenstein rings that are not necessarily complete intersections. This fact can be shown as follows:

\begin{rmk}\label{kanitla} Let $(R, \fm)$ be a one-dimensional Gorenstein local ring, and let $I$ be an Ulrich ideal of $R$ that contains a parameter ideal $\fq$ as a reduction (in particular, $I$ is an $\fm$-primary ideal). Note $I$ is generated by two elements; see \cite[2.6(b)]{GotoUlrich}. Since $I^2=qI$, there is an exact sequence $0 \to \fq/I^2 \to I/I^2 \to I/\fq \to 0$, where $I/I^2 \cong (R/I)^{\oplus 2}$, $\fq/I^2\cong R/I$ and $I/\fq  \cong R/I$. Thus the multiplicity of $I$ equals to $2 \cdot \len_R(R/I)=\len_R(I/I^2)$. Hence \cite[2.3]{ooishi} implies that $I$ is a self-dual $R$-module, i.e., $I \cong I^{\ast}$. Since $I$ contains a non zero-divisor on $R$ and $I$ is not principal, we see $I\otimes_RI^{\ast}\cong I\otimes_RI$ has torsion.
\end{rmk}

\section{Appendix A: on the vanishing of theta invariant}

Recall that, if $R$ is a one-dimensional reduced hypersurface ring, then $\theta^R(M,N)$ is defined and vanishes for all modules $M, N\in \md R$, either of which has rank. Since Theorem \ref{Main} relies upon the vanishing of $\theta$ pairing, we would like to find out whether $\theta$ can vanish nontrivially. More precisely, we would like to find out whether there is a one-dimensional reduced hypersurface ring $R$, and modules $M$ and $N$ over $R$ -- neither of which has rank -- such that $\theta^R(M,N)=0$. We were unable to find an example (or a result) from the literature that addresses our query. The aim of this section is to record such an example suggested to us by Hailong Dao; see Example \ref{Omoshiroi}. First, in \ref{Giro}, we will record a related fact that was shown to us by Mark Walker: over one-dimensional reduced local rings $R$, a module $M\in \md R$ has rank if and only if its class is zero in $\overline{G}(R)_{\QQ}$. A similar result that makes use of $\theta$ pairing is established in \cite[3.3]{Da1} for hypersurface rings.

\begin{chunk} \label{Giro} Let $R$ be a one-dimensional Cohen-Macaulay ring, and let $M\in \md R$. 
\begin{enumerate}[\rm(i)]
\item There exists a rational number $r$ such that $\len_{R_{\p}}(M_{\p})=r \cdot \len_{}(R_{\p})$ for each $p\in \Ass(R)$ if and only if $[M]=0$ in $\overline{G}(R)_{\QQ}$.
\item Assume $R$ is reduced. Then $M$ has rank if and only if $[M]=0$ in $\overline{G}(R)_{\QQ}$.
\end{enumerate}
\end{chunk}

\begin{proof}
(i) Let $\Ass(R)=\{p_1, \ldots, p_n\}$ be the set of all minimal (associated) primes of $R$. Note that $\G(k)=\ZZ \cdot [k]$ and $\G(R_{p_j}) = \ZZ \cdot [k(p_j)]$, where $k(p_j)$ is the residue field of $R_{p_{j}}$, for all $j=1, \ldots, n$.

There is a right exact sequence of the form:
\begin{equation}\tag{\ref{Giro}.1}
G(k) \stackrel{\rm{\alpha}}{\lra} G(R) \stackrel{\rm{\beta}}{\lra} \bigoplus^{n}_{j=1} G(R_{p_j})  {\lra} 0.
\end{equation}
Here $\alpha$ is the natural inclusion with $\alpha([k])=[k]$ and $\beta([M])=\big(\len_{R_{\p_{j}}}(M_{\p_{j}})[k(p_j)] \big)_j$. In (\ref{Giro}.1), by identifying $\G(k)$ with $\ZZ$, and $\bigoplus^{n}_{j=1} G(R_{p_j})$ with $\ZZ^{\oplus n}$, we obtain another right exact sequence of the form
\begin{equation}\tag{\ref{Giro}.2}
\ZZ \stackrel{\rm{\alpha}}{\lra} G(R) \stackrel{\rm{\beta}}{\lra} \ZZ^{\oplus n} {\lra} 0,
\end{equation}
where $\alpha(1)=[k]$ and $\beta([M])=\big( \len_{R_{\p_{j}}}(M_{\p_{j}})  \big)_j$.
Applying $-\otimes_{\ZZ}\QQ$ to (\ref{Giro}.2), we see there is a right exact sequence of the form:
\begin{equation}\tag{\ref{Giro}.2}
\QQ \stackrel{\rm{\alpha\otimes_{}1}}{\lra} G(R)_{\QQ} \stackrel{\rm{\beta\otimes_{}1}}{\lra} \QQ^{\oplus n} {\lra} 0.
\end{equation}
Here $\alpha \otimes 1(1)=[k]$, which is zero in $G(R)_{\QQ}$. Hence $\alpha\otimes_{}1$ is the zero map so that $\beta\otimes_{}1$ is an isomorphism.

Consequently $[M]=0$ in $\overline{G}(R)_{\QQ}$ if and only if $[M]=r \cdot [R]$ for some rational number $r$ if and only if $\beta\otimes_{}1([M])= r \cdot \beta\otimes_{}1([R])$ if and only if $\len_{R_{p_j}}(M_{p_j}) = r \cdot \len(R_{p_j})$ for all $j=1, \ldots, n$.

(ii) If $M$ has rank, then $[M]=0$ in $\overline{G}(R)_{\QQ}$; see Remark \ref{rmkyeni}. To see the converse, let $p \in \Ass(R)$. Then, by (i), we have $\len_{R_{\p}}(M_{\p})=r \cdot \len_{}(R_{\p})$ for some rational number $r$. Since $M_{\p}\cong R^{\oplus n}_{\p}$ for some positive integer $n$, we see $n=r$ and hence $M$ has rank $r$.
\end{proof}

The next example shows that the conclusion of \ref{Giro}(ii) can fail if $R$ is not reduced.  It also shows that Conjecture \ref{HWC} would fail if the module $M$ in question does not have rank, even if $[M]=0$ in $\overline{G}(R)_{\QQ}$.

\begin{eg} Let $R=\CC[\![x,y]\!]/(x^2)$ and let $M=R/(x)$. Then $M \cong M^{\ast}$, and so $M$ is torsion-free. The exact sequence $0 \to M \to R \to M \to 0$ implies that $[M]=0$ in $\overline{G}(R)_{\QQ}$. Moreover, $M$ is not locally free on $\Ass(R)$; in particular $M$ does not have rank. Note also that $\Tor^R_i(M,M^{\ast})\cong M$ for all $i\geq 0$, and hence $\len(\Tor^R_i(M,M^{\ast}))=\infty$ for all $i\geq 0$.
\end{eg}




Here is an example we seek on the vanishing of $\theta$ invariant.
		
\begin{eg} \label{Omoshiroi} Let $R=\CC[\![x,y]\!]/(xy(x-y))$, $M= R/(x)$ and $N = M \oplus R/(y) \oplus R/(y)$. Then $R$ is a one-dimensional reduced hypersurface ring, and $M, N\in \md R$ are non-free, torsion-free modules. 

The minimal free resolution of $M$ is given by:
$$F=\xymatrix{\cdots \ar[r]^{(x-y)y} & R \ar[r]^{x} &  R \ar[r]^{(x-y)y}& 
R  \ar[r]^{x} & R \ar[r] & 0}.$$
Thus $\Tor_1^R(M,M) \cong k[\![y]\!]/(y^2)$ and $\Tor_2^R(M,M)=0$ so that $\theta^R(M,M)=-2$. 
Similarly one can check $\Tor_1^R(R/(y),R/(y)) \cong k[\![x]\!]/(x^2)$ and $\Tor_2^R(R/(y),R/(y))=0$. So it follows $\theta^R(R/(y),R/(y))=-2$. Tensoring $F$ with $R/(y)$, we see $\Tor_2^R(M,R/(y))\cong k$ and $\Tor_1^R(M,R/(y))=0$. This yields $\theta^R(M,R/(y))=1$. 

Therefore we have $\theta^R(N,N)=-6$ and $\theta^R(M,N)=\theta^R(M,M)+2\theta^R(M,R/(y))=0$.
Note that, since $\theta^R(M,M)\neq 0$ and $\theta^R(N,N)\ne 0$, neither $M$ nor $N$ has rank.
\end{eg}

\begin{rmk} \label{rmk74} It seems interesting that, contrary to Example \ref{Omoshiroi}, over certain reduced hypersurface rings, $\theta(M,N)$ can vanish only when $M$ and $N$ have rank. For example, if $R=\CC[\![x,y]\!]/(xy)$ and $M,N\in \md R$, then one can check that $\theta^R(M,N)$ vanishes if and only if $M$ and $N$ both have rank. Note, by \ref{Giro}, one concludes for this particular hypersurface ring $R$, and modules $M$ and $N$ that, $\theta^R(M,N)=0$ if and only if $[M]=[N]=0$ in $\overline{G}(R)_{\QQ}$ if and only if $M$ and $N$ both have rank.
\end{rmk}

\section{Appendix B: some examples of torsion-free tensor products} In this section we recall that Conjecture \ref{HWC} may fail if one considers the tensor product $M\otimes_{R}M$ instead of $M\otimes_{R}M^{\ast}$. Huneke and Wiegand showed that, if $R$ is a one-dimensional local domain that is not Gorenstein, then there exists a torsion-free module $M\in \md R$ such that $M$ is not free and $M\otimes_{R}M$ is torsion-free; see the proof of \cite[4.7]{HW1}. However their argument seems to not yield an explict example of such a module $M$. Building on the proof of Huneke and Wiegand, we will point out how to construct torsion-free modules $M$ with rank such that $M\otimes_{R}M$ is torsion-free over certain one-dimensional local rings $R$. 

\begin{chunk} \label{Appendix2main} Let $R$ be a one-dimensional Cohen-Macaulay local ring with a canonical module $\omega$. Set $M=\Tr\Omega\Tr\Omega\omega$. If $R$ is generically Gorenstein but not Gorenstein, then $M$ is a non-free, torsion-free $R$-module with rank such that $M \otimes_{R} M$ is torsion-free.
\end{chunk}

\begin{proof} It follows from \cite[2.21]{AuBr} that there is an exact sequence of the form:
\begin{equation}\tag{\ref{Appendix2main}.1}
0 \to F \to M\oplus G \to \omega \to 0,
\end{equation}
where $F, G \in \md R$ are free modules. In particular, $M$ and $M^{\ast}$ are torsion-free modules such that $M$ has rank and $M^{\ast}$ is nonzero. As syzygy modules are torsionless, we have $\Ext^1_R(M,R)=0$. It follows $M\otimes_{R}\omega $ is torsion-free, and the sequence (\ref{Appendix2main}.1) does not split; see \cite[B.4]{Tachikawa} and \cite[2.5]{AR2017}. Now tensoring (\ref{Appendix2main}.1) with $M$, we see $M\otimes_{R}M$ is torsion-free. 
\end{proof}


Modules yielding torsion-free tensor products as in \ref{Appendix2main} can also be obtained without appealing to the short exact sequence involving the transpose. Such a module can be realized as the \emph{pushforward} of the first syzgy of the canonical module of the  ring $R$. We observe this below by including a few additional details to the argument of \cite[4.7]{HW1}.


\begin{chunk} (see \cite[4.7]{HW1}) \label{PFrmk} Let $M\in \md R$ and let $\pi: F \to M^{\ast}$ be a minimal free presentation of $M^{\ast}$. Denote $\mu: M \to F^*$ by the composition of the natural map $\delta_M:M \to M^{**}$ and $\pi^*: M^{**} \to F^*$. Then $\mu ^*$ is surjective, and the cokernel of $\mu$, denoted by $\PF(M)$, is called the pushforward of $M$ (pushforward is unique up to free summands; see, for example, \cite{Ce}). 

Now assume $M$ is torsionfree and $\Ext_R^1(M, R)\neq 0$. Take a minimal generating set $\alpha_1, \ldots, \alpha_t$ of $\Ext_R^1(M, R)$. Then each $\alpha_i$ represents a short exact sequence of the form $0 \to R \to N_i \to M \to 0$. Let $\alpha: 0 \to R^{\oplus t} \to N \to M \to 0$ be a pullback of the 
short exact sequence $\oplus_{i=1}^t \alpha_i: 0 \to R^{\oplus t} \to \oplus_{i=1}^tN_i \to M^{\oplus t} \to 0$ by the diagonal map $\Delta: M \to M^{\oplus t}$. 
Then $\alpha = (\alpha_1, \cdots, \alpha_t) \in \Ext_R^1(M, R^{\oplus t}) \cong \Ext_R^1(M, R)^{\oplus t}$.
Next consider the induced exact sequence: $0 \to M^* \to N^* \to (R^t) ^* \xrightarrow{\alpha} \Ext_R^1(M, R) \to \Ext_R^1(N, R) \to 0$. Since the map $(R^t) ^* \xrightarrow{\alpha} \Ext_R^1(M, R)$ is surjective, we see that $\Ext_R^1(N, R)=0$. Thus, in the following pullback diagram, $W$, being a direct sum of $R^{\oplus s} $ and $R^{\oplus t}$, is free. So the vanishing of $\Ext_R^1(N, R)$ shows that $N=\PF(\Omega M)$.
\begin{center}
\pushQED{\qed} 
$$
\begin{CD}
@.  @. 0 @. 0 @. \\
@. @. @AAA @AAA \\
0 @>>> R^{\oplus t} @>>> N @>>> M @>>> 0 \\
@. @| @AAA @AAA @. \\
0 @>>> R^{\oplus t} @>>> W @>>> R^{\oplus s} @>>> 0 \\
@. @. @AAA @AAA @. \\
@.  @. \Omega M @= \Omega M @.\\
@. @. @AAA @AAA \\
@. @. 0 @. 0
\end{CD}
$$
\end{center}
\vspace{0.1in}
Now, if $R$ is as in Proposition \ref{Appendix2main} and $M=\omega$, it follows that $N\otimes_RN$ is torsion-free. \qedhere
\popQED
\end{chunk}

The observation in \ref{Appendix2main} raises the following question; an affirmative answer to it yields a counterexample to Conjecture \ref{HWC}.

\begin{ques} \label{bilinmez} Is there a one-dimensional, generically Gorenstein, Cohen-Macaulay local ring $R$ with a canonical module $\omega$ such that $\omega \ncong R$ and $(\Tr\Omega\Tr\Omega \omega)^{\ast} \cong \Tr\Omega\Tr\Omega \omega$?
\end{ques}

In the next example we record a non-free, torsion-free module $N$ over a one-dimensional local domain $R$, where $N \otimes_{R} N$ is torsion-free, but $N \otimes_{R} N^{\ast}$ has torsion. 

\begin{eg} Let $R=\CC[\![t^3,t^4, t^5]\!]=\CC[\![x,y,z]\!]/(y^2-xz, x^3-yz, x^2y-z^2)$. Then $R$ is a one-dimensional local domain which is not Gorenstein, and let $N\in \md R$ be the module given by the following exact sequence:
$$\xymatrix@C=0.9cm{ 
R^{\oplus 3} \ar[rrr]^-{\left[\begin{array}{ccc} -y & x & z \\
x^2 & -z & -xy  \\ -z & y & x^2  \end{array}\right]} & 
& &  R^{ \oplus 3}   \ar[r]
 & N \ar[r]
&0.}$$
One can check, for example, by using Macaulay 2 \cite{MC2}, that both $N$ and $N \otimes_{R} N$ are torsion-free. Moreover, it follows that $N\otimes_R N^{\ast}$ has torsion; see \cite[3.6]{HSW}.
\pushQED{\qed} 
 \qedhere
\popQED	
\end{eg}

\section{Acknowledgments}
The author would like to thank Tokuji Araya, Hailong Dao, Mohsen Gheibi, Shiro Goto, Li Liang, Hiroki Matsui, W. Frank Moore, Greg Piepmeyer, Arash Sadeghi, Ryo Takahashi, Naoki Taniguchi, Mark Walker and Yongwei Yao for their valuable comments, suggestions and useful discussions related to the manuscript.

\end{document}